\documentclass{amsart}

\usepackage{amsfonts,amssymb}
 \usepackage{xspace,xcolor}
 \usepackage{graphicx,enumerate}
 \usepackage{tikz}

 \usepackage{color}
\usepackage[colorlinks,citecolor=blue]{hyperref}

\theoremstyle{plain}
\newtheorem{theorem}{Theorem}[section]

\newtheorem{lemma}[theorem]{Lemma}
\newtheorem{corollary}[theorem]{Corollary}

\newtheorem{proposition}[theorem]{Proposition}

\newtheorem{prop-defn}[theorem]{Proposition-Definition}

\newtheorem*{theorem:main}{Main Theorem} 
\theoremstyle{definition}

\newtheorem{remark}[theorem]{Remark}

\newtheorem*{remark*}{Remark}
\newtheorem*{remarks*}{Remarks}

\newcommand{\diam}{\rm diam}
\newcommand{\Teich}{\mathcal T}

\newcommand{\C}{\mathcal C}

\newcommand{\SL}{\rm SL}

\title{Pseudo-Anosov homeomorphisms not arising from branched covers.}
\author{Christopher J. Leininger}
\author{Alan W. Reid}
\address{\newline Department of Mathematics,
\newline University of Illinois Urbana-Champaign, 
\newline Urbana, IL 61801, USA}
\email{clein@math.uiuc.edu}
\address{\newline Department of Mathematics,
\newline Rice University, 
\newline Houston, TX 77005, USA}
\email{ alan.reid@rice.edu}

\thanks{The authors gratefully acknowledge support from NSF grants DMS-1510034 (CL) and DMS-1755177(AR)}

\begin{document}

\maketitle

\begin{abstract} In this paper we provide a negative answer to a question of Farb about the relation between the algebraic degree of the stretch factor of a pseudo-Anosov homeomorphism and the genus of the surface on which it is defined.
\end{abstract}

%%%%%%%%%%%%%%%%%%%%%%%%%%%%%%%%%%%%
%%%%%%%%%%%%%%%%%%%%%%%%%%%%%%%%%%%%
\section{Introduction}
%%%%%%%%%%%%%%%%%%%%%%%%%%%%%%%%%%%%
%%%%%%%%%%%%%%%%%%%%%%%%%%%%%%%%%%%%

Let $S = S_g$ be a closed, connected surface of genus $g$.  A pseudo-Anosov homeomorphism $f \colon S \to S$ is a {\em virtual lift} if there exists a branched cover $p \colon S \to \Sigma$ with degree $\deg(p) >1$ and a pseudo-Anosov $\phi \colon \Sigma \to \Sigma$ so that $\phi$ lifts to a power of $f$ by $p$; that is, there exists $n > 0$ so that $pf^n = \phi p$.  We say that {\em $f^n$ is a lift of $\phi$ via $p$}.

Franks-Rykken \cite{FranksRykken} showed that if $f \colon S \to S$ is a pseudo-Anosov (with orientable stable/unstable foliations), $g \geq 2$, and if the stretch factor $\lambda(f)$ is a quadratic irrational, then $f$ is a virtual lift---in fact, the branched cover is over a torus $p \colon S \to \Sigma$ (cf.~Gutkin-Judge \cite{GutkinJudge} and Kenyon-Smillie \cite{KenyonSmillie}).  In $2004$, Farb asked (see \cite{Strenner})
if a version of this is true when the degree of the stretch factor was greater than $2$.  Specifically, he asked if there exists a function $h \colon \mathbb N \to \mathbb N$ so that if $f \colon S_g \to S_g$ is a pseudo-Anosov homeomorphism, the degree of $\lambda(f)$ over $\mathbb Q$ is at most $d$, and $g \geq h(d)$, is $f$ a virtual lift?  Here we prove that the answer is `no'.

\begin{theorem:main} For any even $d \geq 4$ and all $g \geq \tfrac{d}2+2$, there exists pseudo-Anosov homeomorphisms $f_{g,d} \colon S_g \to S_g$ with orientable stable/unstable foliations and $\lambda(f_{g,d})$ of degree $d$ over $\mathbb Q$, so that $f_{g,d}$ is not a virtual lift.
\end{theorem:main}

%%AR added the Yazdi comment and the comments about Bestvina-Fujiwara and Strenner
%
We note that simultaneously and independently, M. Yazdi \cite{Yaz} has also answered Farb's question in the negative. In \cite{Yaz} he shows that for all $g\geq 3$, there are 
pseudo-Anosov maps $f_g:S_g\rightarrow S_g$ so that $\lambda(f_g)$ has degree $6$ and at most finitely many of them can be virtual lifts.  The method of proof is different from that given here.

We also mention the related results  \cite[Lemma 6.2]{BF} and \cite[Corollary 1.4]{Strenner} that both describe conditions which guarantee that a pseudo-Anosov is not a virtual lift. In the former case no control on the stretch factor is given, and in the
latter the stretch factors have degree $6g-6$ (the maximal possible degree).
%
%
%%AR add acknowledgment 

We complete the Introduction by briefly describing the idea of the proof of the Main Theorem.  The pseudo-Anosov homeomorphisms are constructed as products of high powers of Dehn twists.  The twisting curves and powers are chosen in such a way that we can apply Strenner's results from \cite{Strenner} to ensure that the stretch factor has the appropriate degree.  To prove that the homeomorphisms are not virtual lifts, we analyze the flat metrics defining the associated Teichm\"uller axes.  Appealing to work of Rafi \cite{Rafi1}, Minsky \cite{MinskyShort2}, and Brock-Canary-Minsky \cite{BCM-ELC2}, we prove that for carefully chosen twisting curves, there is a biinfinite collection of simple closed curves that are ``characteristic'' for the pseudo-Anosov.  These characteristic curves are described in terms of Euclidean cylinder neighborhoods with respect to the flat metrics, and if a pseudo-Anosov homeomorphism is a virtual lift, we prove that they must project to the quotient surface in a very specific way.  The proof is completed by choosing the twisting curves so that the associated biinfinite sequence of curves cannot project to any nontrivial quotient surface in that way.\\

\noindent{\bf Acknowledgement:} {\em The authors wish to thank the organizers of the Oberwolfach Workshop "Surface Bundles" in December $2016$ for their invitation to attend the workshop and where this work
started.  We also wish to thank the organizers of the 3d GEAR Network Retreat, Stanford August 2017 where this work was largely completed. The authors would also like to thank Bal\'azs  Strenner for comments on an earlier version of the paper.}

%%%%%%%%%%%%%%%%%%%%%%%%%%%%%%%%%%%%
%%%%%%%%%%%%%%%%%%%%%%%%%%%%%%%%%%%%
\section{Surfaces, curves, and annular projections}
%%%%%%%%%%%%%%%%%%%%%%%%%%%%%%%%%%%%
%%%%%%%%%%%%%%%%%%%%%%%%%%%%%%%%%%%%

Suppose $S$ is any orientable hyperbolic surface of finite topological type.
%%AR added of finite topological type.  
Here we collect a few facts about curve complexes and subsurface projections.  See \cite{MM1} and \cite{MM2} for more details.

The {\em curve graph of $S$}, $\C(S)$, is the simplicial complex whose vertex set $\C^{(0)}(S)$ is the set of isotopy classes of essential simple closed curves on $S$.  A pair of isotopy classes determine an edge if and only if they can be realized disjointly on $S$---equivalently, the geodesic representatives with respect to the hyperbolic metric are disjoint.   We make $\C(S)$ into a geodesic metric space by declaring each edge to have length $1$.  According to \cite{MM1}, $\C(S)$ is $\delta$--hyperbolic.

If $Y$ is an annulus, we define the {\em curve graph of $Y$}, $\C(Y)$, in a similar fashion: the vertex set consists of isotopy classes of essential arcs in $Y$, where isotopies are required to fix the boundary pointwise.  Edges connect isotopy classes when there are representatives with disjoint interiors, and we similarly make $\C(Y)$ into a geodesic metric space.  

The curve graphs of annuli arise from annular subsurfaces of $S$ as follows.  Given an essential annulus $Y \subset S$, there is a corresponding covering space $\widetilde Y \to S$.  The ideal boundary of the universal covering $\mathbb H^2 \to S$ determines an ideal boundary of $\widetilde Y$, and we let $\overline Y$ denote $\widetilde Y$ together with its ideal boundary, making $\overline Y$ into a compact surface with boundary.  Given a vertex $\alpha$ of $\C(S)$, representing $\alpha$ by its hyperbolic geodesic representative, we let $\widetilde \alpha$ denote the union of the arcs in the preimage of $\alpha$ in $\overline Y$.  We define $\pi_Y(\alpha)$ to be the union of the components of $\widetilde \alpha$ which are essential in $\overline Y$ (together with their ideal endpoints); that is, the components with an endpoint on each boundary component of $\bar Y$.  We view $\pi_Y(\alpha)$ as subset of $\C(Y)$.  Note that $\pi_Y(\alpha) \subset \C(Y)$ has diameter $1$ (any
  two components are disjoint).  Given two curves $\alpha,\beta \in \C(S)$, if $\pi_Y(\alpha)$ and $\pi_Y(\beta)$ are both nonempty, we define
\[ d_Y(\alpha,\beta) = \diam(\pi_Y(\alpha) \cup \pi_Y(\beta)).\]
%%AR added a missing )
%
With this definition, given any three curves with non-empty projections, $d_Y$ satisfies a triangle inequality.  See \cite{MM2} for more on subsurface projections, and definitions for other types of subsurfaces.

The core curve of $Y$ is an essential simple closed curve $\gamma$ in $S$ and every essential simple closed curve is the core curve of an essential annulus.  We sometimes write $\C(\gamma)$, $\pi_\gamma$, and $d_\gamma$ instead of $\C(Y)$, $\pi_Y$, and $d_Y$, respectively.  We have $\pi_\gamma(\alpha) \neq \emptyset$ if and only if the geometric intersection number, $i(\alpha,\gamma)\neq 0$.
%%AR rewrote the last sentence

One of the key features of subsurface projections is the following {\em Bounded Geodesic Image Theorem} in the case of annuli.

\begin{proposition} \label{P:BGI} There exists a constant $M > 0$ with the following property.  If $\alpha,\beta$ are two curves in $\C(S)$ and $d_\gamma(\alpha,\beta) > M$, then the geodesic from $\alpha$ to $\beta$ contains a vertex $\delta$ so that $i(\delta,\gamma) = 0$, and hence $\delta$ is adjacent to $\gamma$ in $\C(S)$.
\end{proposition}

The following is a special case of the Behrstock Inequality \cite{Behr} for annuli that we will need.

\begin{proposition}  \label{P:Behrstock} Suppose $\alpha,\beta,\gamma$ are three simple closed curves on $S$ that pairwise intersect.  If $d_\gamma(\alpha,\beta) \geq 10$ then $d_\alpha(\gamma,\beta) \leq 3$.
\end{proposition}

This version with explicit constants is proved by Mangahas in \cite{joh1,MangahasRecipe}.

\section{Teichmuller geodesics and Euclidean cone metrics}
%%%%%%%%%%%%%%%%%%%%%%%%%%%%%%%%%%%%
%%%%%%%%%%%%%%%%%%%%%%%%%%%%%%%%%%%%

A pseudo-Anosov homeomorphism $f \colon S \to S$ preserves a Teichm\"uller geodesic axis defined by a unit area non-positively curved Euclidean cone metric $q_0$ for which the stable and unstable foliations $\mu^\pm$ are orthogonal, geodesic foliations.  Furthermore, in preferred coordinates $\mu^\pm$ are horizontal and vertical, respectively, and the transverse measures are given by horizontal and vertical variation, respectively.  The different points along the axis are conformal structures of Euclidean cone metrics $q_t$ in which the stable and unstable foliations have their transverse measures scaled as $e^t \mu^+, e^{-t} \mu^-$ (maintaining unit area for the Euclidean cone metrics). We call the family of Euclidean cone metrics $Q =  \{q_t\}_{t \in \mathbb R}$ the {\em associated flat metics}.   Note that any two metrics in the family differ by an affine diffeomorphism (away from the cone points).  We write $\ell_{q_t}(\gamma)$ for the $q_t$--length of a curve $\gamma$.

If $f^n$ is a lift of $\phi \colon \Sigma \to \Sigma$ via a branched cover $p \colon S \to \Sigma$, then the associated flat metrics $\Xi = \{\xi_t\}$ for $\phi$ can be chosen so that $q_t = \sqrt{\deg(p)} \, p^*(\xi_t)$  (this scaling is necessary since $q_t$ and $\xi_t$ have unit area).  In this case, we say that $Q =\{q_t\}$ and $\Xi = \{\xi_t\}$ are {\em compatible}.

If $Q = \{q_t\}$ are the flat metrics associated to a pseudo-Anosov on $S$ as above, a {\em $Q$--cylinder} or {\em flat cylinder for $Q$} (or just flat cylinder, if $Q$ is understood) is an annulus $Y \subset S$ so that the path metric on $Y$ coming from some $q_t  \in Q$ makes $Y$ into a Euclidean product $I \times S^1$, where $I$ is an interval (we allow the possibility that $Y$ is only embedded on its interior, but still write $Y \subset S$).  Note that if the metric on $Y$ is a Euclidean product for some $q_t \in Q$, then it is for all $q_t \in Q$ (and any two such metrics differ by affine diffeomorphism).   The $q_t$--modulus of a flat cylinder $Y \subset S$, denoted $M(Y,q_t)$, is the ratio of the height to circumference, and $M(Y,Q) = \max\{M(Y,q_t) \mid t \in \mathbb R\}$ is the maximum modulus.  If $\gamma \subset S$ is a two-sided simple closed curve, there is a (possibly degenerate) maximal flat cylinder $Y_\gamma \subset S$ whose core curves are isotopic to $\gamma$, and we set $M(\gamma,q_t) = M(Y_\gamma,q_t)$ and $M(\gamma,Q) = M(Y_\gamma,Q)$.

We say that $\gamma$ is a {\em $Q$--cylinder curve} if $M(\gamma,Q) > 0$.  In this case, there is a unique $t_\gamma \in \mathbb R$, called the {\em balance time} of $\gamma$, so that the $q_t$--length $\ell_{q_t}(\gamma)$ is minimized at $t_\gamma$ and $M(\gamma,q_{t_\gamma}) = M(\gamma,Q)$.  Indeed, $t_\gamma$ is the unique time for which the core geodesics of the Euclidean cylinder make angle $\frac{\pi}4$ with the vertical and horizontal foliations, and
\[ \ell_{q_t}(\gamma) = \ell_{q_{t_\gamma}}(\gamma) \cosh^{\frac12}(2(t-t_\gamma)).\]
The following is an easy case of a result of Rafi \cite{Rafi1}.

\begin{proposition} \label{P:flat annuli iff twisting} Suppose $f \colon S \to S$ is a pseudo-Anosov homeomorphism, $Q$ is the associated family of flat metrics, and $\mu^\pm$ are the stable and unstable foliations.  If $d_\gamma(\mu^+,\mu^-) > 4$, 
%%AR reordered +,- to be consistent
then $\gamma$ is a $Q$--cylinder curve.  In general, if $\gamma$ is a $Q$--cylinder curve, then
\[ \left| M(\gamma,Q) - \frac{d_\gamma(\mu^+,\mu^-)}2 \right| \leq 3 \]
\end{proposition}
\begin{proof}[Sketch of proof.] Choose lifts of leaves $\delta^+$ of $\mu^+$ and $\delta^-$ of $\mu^-$ to the annular cover $\widetilde Y_\gamma$ of $S$ so that
\[ |d_\gamma(\mu^+,\mu^-) - i(\delta^+,\delta^-) |  \leq 2. \]
By Lemma 3.8 of \cite{Rafi1}, all but at most $2$ of the intersection points between $\delta^+$ and $\delta^-$ must occur in (the lift of) the maximal flat Euclidean cylinder $Y_\gamma$ (which is thus nonempty if $d_\gamma(\mu^+,\mu^-) > 4$).

All that remains is to prove that $M(\gamma,Q)$ is within $2$ of $\frac{I}2$, where $I = i(\delta^+,\delta^-)$.  For this, note that $\delta^+$ and $\delta^-$ cut $Y_\gamma$ into squares (or parts of squares near $\partial Y_\gamma$) whose sides make angle $\frac{\pi}4$ with the core curves.  In each square there is a diagonal that closes up to a core curve, and hence has length $L = \ell_{q_{t_\gamma}}(\gamma)$.  On the other hand, there are geodesics running orthogonally from one boundary component to the other that cut roughly half the squares in diagonals opposite those that define core curves.  Considering the two extreme cases (when there are two intersection points outside $Y_\gamma$ and when all intersection points are inside $Y_\gamma$ and as far as possible from $\partial Y_\gamma$), we see that the distance $D$ between boundary components satisfies
\[ \frac{L(I-3)}2 \leq D \leq \frac{L(I+1)}2. \]
Since $M(\gamma,Q) = \frac{D}{L}$, this implies $\tfrac{I-3}2 \leq M(\gamma,Q) \leq \frac{I+1}2$, completing the proof.
\end{proof}

The proof of our main theorem will rely on understanding how $Q$--cylinders in $S$ are mapped down to $\Sigma$.  The images need not be cylinders, but with some additional mild assumptions, they are very well behaved.  A {\em Euclidean half-pillowcase} is the quotient of a Euclidean cylinder $S^1 \times [-T,T]$ by the group generated by the involution $\tau(e^{i\theta},t) = (e^{-i \theta},-t)$.  Considering a fundamental domain for this action, we can equivalently describe this as the Euclidean orbifold obtained by gluing a component of the boundary of a Euclidean cylinder $S^1 \times [0,T]$ to itself by the map $(e^{i \theta},0) \sim (e^{-i \theta},0)$.  Topologically, a half-pillow case is a disk with two marked points.  The two marked points are cone points with cone angle $\pi$ and there is a geodesic segment, the {\em core segment}, connecting those points whose complement is itself a half-open Euclidean cylinder.  We will refer to the modulus of the complementary Euclid
 ean cylinder as the {\em modulus of the half-pillowcase}.

\begin{lemma} \label{L:cylinders into orientable} Suppose $\Sigma$ is an orientable surface and $\phi \colon \Sigma \to \Sigma$ a pseudo-Anosov homeomorphism with associated flat metrics $\Xi = \{\xi_t\}$.  Assume that the only marked points of $\Sigma$ are cone points of $\xi_t$ with cone angle $\pi$ and that $\Sigma$ is not a torus or a sphere with four marked points.  Let $h \colon Y \to \Sigma$ denote a map of an open Euclidean cylinder into $\Sigma$ which for some $\xi_t \in \Xi$, is a local isometry away from a finite number of branched points.
Then either $h(Y)$ is a Euclidean cylinder in $\Sigma$ and $h$ is a covering map onto its 
%%AR changed typo its image
image or else $h(Y)$ is a Euclidean half-pillowcase.  In either case, $M(h(Y),\xi_t) \geq \frac{M(Y)}2$.
\end{lemma}
\begin{proof} First suppose that there are no branch points in $Y$.  In this case, each core geodesic of $Y$ maps to a geodesic.  Since the holonomy of $\xi_t$ is $\{\pm I\}$, it follows that these geodesics are simple.  Since $\Sigma$ is orientable, the sub-cylinder between the two core geodesics provides an isotopy from one to the other.  Suppose two of the 
%%AR added the
core geodesics $\alpha,\beta$ of $Y$ map to the same simple closed geodesic in $\Sigma$.  Orient both $\alpha$ and $\beta$ in the same direction coming from the annulus (so the isotopy between them is orientation preserving).  Again, because $\Sigma$ is orientable, $\alpha$ and $\beta$ must map to the same oriented curve.  Since the sub-cylinder between $\alpha$ and $\beta$ lies on different sides of these two curves (each are two-sided curves), it follows that image of the cylinder lies on both sides of the image.  Thus, we can identify $\alpha$ and $\beta$ in the sub-cylinder producing a torus which maps locally isometrically to $\Sigma$.  Therefore $\Sigma$ is a flat torus, which is a contradiction.  Thus, no two core geodesics of $Y$ are sent to the same curve, and it follows that $h(Y)$ is a cylinder, foliated by the images of the core geodesics.  Since $h$ restricts to a covering map from each core geodesic onto its image, it follows that $h$ restricts to a covering map
  from $Y$ onto its image.

If $h$ is nontrivially branched, let $\zeta \in Y$ be a point at which $h$ is branched, and note that $h(\zeta)$ must be a cone point of angle $\pi$.  Let $\alpha$ be a core geodesic through $\zeta$, and observe that this must project to a geodesic segment between a pair of cone points with angle $\pi$.  Geodesics sufficiently close to $h(\alpha)$ project to geodesic surrounding $h(\alpha)$, and hence a neighborhood of $\alpha$ maps down to a Euclidean half-pillowcase.  If there is another core geodesic $\beta \neq \alpha$ of $Y$ that also contains a branch point, then choose one that is closest to $\alpha$, and observe that the Euclidean cylinder between $\alpha$ and $\beta$ contains no cone points, and can be glued together to make a sphere with four cone points which maps locally isometrically (away from cone points) onto $\Sigma$ (this is similar to the case of no branch points where we showed that $\Sigma$ was the image of a flat torus).  The only orientable Euclidean co
 ne surfaces with holonomy $\{\pm I\}$ which is the image of a locally isometric map of the sphere with four cone points is the sphere with four cone points, and so $\Sigma$ is a sphere with four cone points, a contradiction.   It follows that there is only one geodesic $\alpha$ which contains branch points.  The sub-cylinders on either side of $\alpha$ map to $\Sigma$ without branched points, so by the previous paragraph, these cover cylinder.  Thus $h(Y)$ is a Euclidean half-pillowcase, namely the union of the half-pillowcase neighborhood of the image of $\alpha$, together with these two cylinders (which share some core geodesics).

If $h \colon Y \to h(Y)$ is a covering map, then the modulus of $h(Y)$ is the modulus of $Y$ times the degree of this covering.  In the two-fold quotient from a Euclidean cylinder to a half-pillowcase, the modulus is reduced by half.  The lower bound on modulus now follows.  This completes the proof.
\end{proof}

\begin{remark} We note that when $h(Y)$ is a Euclidean half-pillowcase, the map $h$ is not necessarily a (branched) covering map from $Y$ to $h(Y)$: the two distances from the core geodesic $\alpha$ to the two boundary components might be different.
\end{remark}

\begin{lemma} \label{L:cylinders into nonorientable} Suppose $\Sigma$ is a nonorientable surface and $\phi \colon \Sigma \to \Sigma$ a pseudo-Anosov homeomorphism with associated flat metrics $\Xi = \{\xi_t\}$.
%Assume that the only marked points are cone points of $\xi_t$ with cone angle $\pi$.  
Let $h \colon Y \to \Sigma$ denote a map of an open Euclidean cylinder into $\Sigma$ which for some $\xi_t \in \Xi$, is a local isometry away from a finite number of branched points.
Further assume that the modulus of $Y$ is strictly greater than $2$.  Then $h(Y)$ is either a Euclidean cylinder or a Euclidean half-pillowcase and $M(h(Y),\xi_t) \geq \frac{M(Y)}2$.
\end{lemma}
\begin{proof} Let $g: \Sigma' \to \Sigma$ denote the orientation double cover, and observe that $h$ lifts to $h' \colon Y \to \Sigma'$; indeed, after puncturing $Y$ at the branch points, $h$ becomes a local diffeomorphism, and then this is a general fact about local diffeomorphisms from an orientable manifold to a non-orientable manifold of the same dimension.  A pseudo-Anosov homeomorphism on a torus or sphere with four marked points cannot be a lift of a pseudo-Anosov homeomorphism of a non-orientable surface: this follows from \cite[Proposition 2.3]{StrennerNonorientable}, for example, where it is shown that lifts of pseudo-Anosov homeomorphisms from a nonorientable surface have stretch factors that are {\em not} Galois conjugates, while stretch factors of pseudo-Anosov homeomorphisms of the torus and sphere with four marked points are quadratic irrational algebraic integers, and hence their inverses are their Galois conjugates.  Therefore, by Lemma~\ref{L:cylinders into orientable}, $h'(Y) \subset \Sigma'$ is either a Euclidean cylinder or half-pillowcase with the required lower bound on modulus.

Since $g$ is a two-fold covering, there is another lift $h'' \colon Y \to \Sigma'$.  We claim that $h'(Y)$ and $h''(Y)$ are disjoint, and hence the restriction of $g$ to $h'(Y)$ is a homeomorphism onto $h(Y)$, which by Lemma~\ref{L:cylinders into orientable} will complete the proof.  The map $h''$ differs from $h'$ by composing with the order two covering map $\sigma \colon \Sigma' \to \Sigma'$, which is orientation reversing.  If $h'(Y) \cap h''(Y) \neq \emptyset$, then there is a point $z$ of $h'(Y)$ for which $\sigma(z) \in h'(Y)$.  If $h'(Y)$ is a half-pillowcase, we can assume that $z$ and $\sigma(z)$ lie in the Euclidean cylinder surrounding the core segment between the cone points.  In this case, we restrict our attention to this Euclidean cylinder, which by our assumption has modulus strictly greater than $1$.  To deal with both cases simultaneously, as a slight abuse of notation, we let $h'(Y)$ and $h''(Y)$ denote these two annuli.

Next choose an oriented orthonormal basis $e_1,e_2$ on $h'(Y)$ so that $e_1$ is tangent to the core curves of $h'(Y)$.  The derivative $d\sigma_z \colon T_z(h'(Y)) \to T_{\sigma(z)}(h'(Y))$ is orientation reversing, hence a reflection.  Since the stable/unstable foliations are preserved by $\sigma$, the line of reflection must be tangent to one of these foliations.  Since these foliations are orthogonal, and neither has closed leaves, we see that the lines of reflection are not spanned by either $e_1$ or $e_2$.  It follows that $\sigma$ must send the core geodesic of $h'(Y)$ through $z$ transverse to the core geodesic through $\sigma(z)$.  Since the modulus of $h'(Y)$ is greater than $1$, the core geodesic is shorter than the distance between the boundary components, which is a contraction.  Therefore, $h'(Y)$ and $h''(Y)$ are disjoint, completing the proof.
\end{proof}

We also need to understand what the preimage of cylinders look like under a branched cover $p \colon S \to \Sigma$.

\begin{lemma} \label{L:preimage of cylinders}  Given $S$ and $d > 0$ there exists $B = B(S,d) > 0$ with the following property.  Suppose that $p \colon S \to \Sigma$ is a branched covering of degree at most $d$, $f \colon S \to S$ a lift of the
%%AR added the
pseudo-Anosov $\phi \colon \Sigma \to \Sigma$, $Q = \{q_t\}$ and $\Xi = \{\xi_t\}$ are the associated, compatible flat metrics, and $Y \subset \Sigma$ is a maximal open $Q$--cylinder with maximal modulus $M(Y,\Xi)$.  Then there is a sub-cylinder $Y_0 \subset Y$ so that $p^{-1}(Y_0)$ is a union of Euclidean cylinders in $S$, each with maximal modulus at least $BM(Y,\Xi)$.% every component $\widetilde Y$ of the preimage of $Y$ contains a Euclidean cylinder $\widetilde Y_0 \subset \widetilde Y$ with modulus at least $B M(\alpha,\Xi)$ so that for every core curve $\widetilde \alpha$ of $\widetilde Y_0$ is sent by $p$ to a core curve of $Y$.
\end{lemma}
\begin{proof}  Fix the metrics $\xi_t$ and $q_t$ at the balance time $t$ of the core curve of $Y$.  By the Riemann-Hurwitz Theorem, there is a bound $b>0$ on the number of branched points of $p$, in terms of $d$ and $\chi(S)$, and we set $B = \frac{1}{d(b+1)}$.  Since $Y$ contains at most $b$ branch points, there are at least $b+1$ open Euclidean sub-cylinders in $Y$ disjoint from the branch points so that the boundaries of the closures in $\Sigma$ are either in the boundary of the closure of $Y$ or else contain a branched point.  The sum of the moduli of these is precisely the modulus of $Y$, and consequently one of them, call it $Y_0$, has modulus at least $\frac{M(Y,\Xi)}{b+1} = \frac{M(Y,\xi_t)}{b+1}$.  The preimage $p^{-1}(Y_0)$ is a Euclidean cylinder and for any component $\widetilde Y_0 \subset p^{-1}(Y_0)$ the restriction of $p$,
\[ p|_{\widetilde Y_0} \colon \widetilde Y_0 \to Y_0,\]
is a covering map of degree at most $d$.  Therefore, $M(\widetilde Y_0,Q) \geq \frac{M(Y,\Xi)}{d(b+1)} = BM(Y,\Xi)$, as required.
\end{proof}

%%%%%%%%%%%%%%%%%%%%%%%%%%%%%%%%%%%%
%%%%%%%%%%%%%%%%%%%%%%%%%%%%%%%%%%%%
\section{Pseudo-Anosovs from Dehn twists.}
%%%%%%%%%%%%%%%%%%%%%%%%%%%%%%%%%%%%
%%%%%%%%%%%%%%%%%%%%%%%%%%%%%%%%%%%%

Suppose $c_1,c_2,\ldots,c_n$ are curves that fill a surface $S = S_g$ with $g \geq 2$ so that $i(c_i,c_{i+1}) \neq 0$ for all $1 \leq i \leq n$ and with $1 \leq i+1 \leq n$ taken modulo $n$.  Let $k_1,k_2,\ldots, k_n \in \mathbb Z$.  Our construction involves analyzing the mapping class
\[ f = T_{c_1}^{k_1} T_{c_2}^{k_2} \cdots T_{c_n}^{k_n} .\]
We first extend the finite sets of curves and integers to infinite sequences $\{c_j\}_{j=1}^\infty$ and $\{k_j\}_{j=1}^\infty$ by setting
\[ c_j = c_{j'} \mbox{ and } k_j = k_{j'}\]
where $1 \leq j' \leq n$ and $j \equiv j'$ modulo $n$.  Then for all $j \geq 1$ set
\[ f_j = T_{c_1}^{k_1} T_{c_2}^{k_2} \cdots T_{c_j}^{k_j}.\]
Observe that for all $m \geq 0$, and $j \geq 0$ we have
\begin{equation} \label{E:periodicity of maps} f_{nm+j} = f^m f_j. \end{equation}
%We can use this to extend our sequence of mapping classes to a biiinfinite sequence: Given any $j \leq 1$, write $j = nm + j'$, with $0 \leq j' \leq n-1$ and set $f_j = f^nf_{j'}$, then (\ref{E:periodicity of maps}) holds for all integers $m$ and $j$.

Now construct a new infinite sequence of curves $\{\gamma_j\}_{j=1}^\infty$ by $\gamma_j = f_j(c_j)$.
For all $j \geq 1$, since $c_j = c_{j+n}$, (\ref{E:periodicity of maps}) implies 
\begin{equation} \label{E:periodic sequence} f(\gamma_j) = f f_j(c_j) = f_{j+n}(c_{j+n}) = \gamma_{j+n}. \end{equation}
Thus, $f$ acts as the $n^{th}$ power of the shift on the sequence $\{\gamma_j\}_{j=1}^\infty$.  Therefore, we can extend the infinite sequence of curves to a biinfinite sequence $\{\gamma_j\}_{j \in \mathbb Z}$ so that (\ref{E:periodic sequence}) holds for all $j \in \mathbb Z$.

%From this, we construct a new biinfinite sequence of curves $\{\gamma_j\}_{j \in \mathbb Z}$ by $\gamma_j = f_j(c_j)$.  Then observe that
%\[ \gamma_j = f_j (c_j) = f^{j-j'} f_{j'}( c_{j'}) \]
%where $1 \leq j' \leq n$ and $j' \equiv j$ modulo $n$.  In particular
%\begin{equation} \label{E:periodic sequence} f(\gamma_j) = f^{j-j'+n} f_{j'}(c_{j'}) = f_{j+n}c_{j+n} = \gamma_{j+n}.\end{equation}
%That is, $\{\gamma_j\}_{j \in \mathbb Z}$ is an $f$--periodic sequence (with period $n$).

\begin{lemma} \label{L:curves with big twist}  Given curves $c_1,\ldots, c_n$ as above, there exists $R > 0$ and $K  >0$ so that if $|k_j| \geq K$ for all $j \geq 1$, then
\begin{enumerate}[(i)]
\item $i(\gamma_i,\gamma_j) \neq 0$ for all $i,j \in \mathbb Z$, $i \neq j$,
\item $|d_{\gamma_\ell}(\gamma_i,\gamma_j)-|k_\ell|| < R$ for all $i,j,\ell \in \mathbb Z$ with $i < \ell < j$.
\item $\{\gamma_i\}$ is an $f$--invariant, uniform quasi-geodesic in the curve complex.
\end{enumerate}
From (iii), 
%%AR changed (3) to (iii)
it follows that $f$ is pseudo-Anosov, and $\{\gamma_j\}_{j \in \mathbb Z}$ is a quasi-geodesic axis.  Consequently, if we let $\mu^\pm$ denote the stable/unstable foliations of $f$, then (ii) implies
\begin{enumerate}
\item[(iv)] $|d_{\gamma_j}(\mu^+,\mu^-) - |k_j|| \leq R + 2$
\end{enumerate}
for all $j \in \mathbb Z$.
\end{lemma}
The meaning of (iii) is that there exists constants $A,B > 0$ so that
\[ \frac1A|i-j| - B \leq d(\gamma_i,\gamma_j) \leq A|i-j| + B.\]
We have avoided cluttering the already lengthy statement by including $A,B$ in the statement.
\begin{proof} We have already established the $f$--invariance of $\{\gamma_j\}$.  In particular, it suffices to prove the statements (i)--(iii) for positive indices.  

First consider a triple of any three consecutive curves $(\gamma_{j-1},\gamma_j,\gamma_{j+1})$.  We want to describe this triple of  curves up to  homeomorphism.  By applying a sufficiently high positive power of $f$, we can assume that $j > 1$.   Then applying $f_{j-1}^{-1}$ to this triple we get
\begin{eqnarray*}
f_{j-1}^{-1}(\gamma_{j-1},\gamma_j,\gamma_{j+1}) & = & f_{j-1}^{-1}(f_{j-1}c_{j-1},f_jc_{j},f_{j+1}c_{j+1})\\
& = & (c_{j-1},T_{c_j}^{k_j}(c_j), T_{c_j}^{k_{j}}T_{c_{j+1}}^{k_{j+1}}(c_{j+1})) \, = \, (c_{j-1},c_{j}, T_{c_{j}}^{k_{j}}(c_{j+1}))
%&= & (c_{j-1},c_{j}, T_{c_{j}}^{k_{j}}(c_{j+1})).
\end{eqnarray*}
Since the sequences $\{c_j\}$ and $\{k_j\}$ are $n$--periodic, we see that up to homeomorphism, any consecutive triple looks like
\[ c_{j-1},c_j,T_{c_{j}}^{k_{j}}(c_{j+1}),\]
for $1 \leq j \leq n$ and the other two indices $1 \leq j-1,j+1 \leq n$ taken modulo $n$.  Since consecutive curves intersect nontrivially, we can apply the triangle inequality for projection distances to obtain
\[ |d_{c_j}(c_{j-1},T_{c_j}^{k_j}(c_{j+1})) - d_{c_j}(c_{j+1},T_{c_j}^{k_j}(c_{j+1})) | \leq d_{c_j}(c_{j-1},c_{j+1}).\]
The right hand side is uniformly bounded by $n$--periodicity, and
\[ |d_{c_j}(c_{j+1},T_{c_j}^{k_j}(c_{j+1}) - |k_j|| \leq 3\]
since the $k^{th}$ power of a Dehn twist acts as translation by $k+1$ on the curve graph of the annulus.
Therefore, taking $R_0 > 0$ to be at least three more than that uniform bound implies
\[ |d_{c_j}(c_{j-1},T_{c_j}^{k_j}(c_{j+1})) - |k_{j}|| \leq R_0.\]
Applying the homeomorphism $f_{j-1}$ to all curves in this inequality, we obtain
\begin{equation} \label{E:twisting nearly k} |d_{\gamma_{j}}(\gamma_{j-1},\gamma_{j+1}) - |k_{j}|| \leq R_0.\end{equation}
We assume $K \geq 2R_0 + 20 + M$, where $M$ is the constant from Proposition~\ref{P:BGI}.  If $|k_j| \geq K$, it then follows that we also have
\[ d_{\gamma_{j}}(\gamma_{j-1},\gamma_{j+1}) \geq 20.\]
In particular, note that $\gamma_{j-1},\gamma_{j},\gamma_{j+1}$ pairwise intersect.\\

\noindent {\bf Claim.} If $i < j$, then $i(\gamma_i,\gamma_j) \neq 0$ and for all $i < \ell < j$, we have $d_{\gamma_i}(\gamma_\ell,\gamma_j) \leq 3$ and $d_{\gamma_j}(\gamma_i,\gamma_\ell) \leq 3$.

\begin{proof}
We prove the claim by induction on $j-i$.  For $j-i=1$ there is no such $\ell$, and so the nonzero intersection number statement is a consequence of the description of triples.
If $j-i = 2$, then the triples description implies $i(\gamma_i,\gamma_j) \neq 0$, and by Proposition~\ref{P:Behrstock}, it follows that $d_{\gamma_i}(\gamma_\ell,\gamma_j) \leq 3$ and $d_{\gamma_j}(\gamma_i,\gamma_\ell) \geq 3$.  These serve as the base case(s).

Now suppose the statement is true whenever the difference in indices is at most $m$, and suppose $j-i = m + 1$.  Without loss of generality, we may assume that $m+1 \geq 3$.  Let $i < \ell < j$ be any index.   Suppose first that
\[ i < \ell - 1 < \ell < \ell + 1 < j. \]
Then by induction $\gamma_\ell,\gamma_{\ell+1},\gamma_j$ pairwise intersect, $\gamma_i,\gamma_{\ell-1},\gamma_\ell$ pairwise intersect, and
\[ d_{\gamma_\ell}(\gamma_{\ell+1},\gamma_j) \leq 3 \quad \mbox{ and } \quad d_{\gamma_\ell}(\gamma_i,\gamma_{\ell-1}) \leq 3.\]
By the triangle inequality, we have
\[ d_{\gamma_\ell}(\gamma_i,\gamma_j) \geq d_{\gamma_\ell}(\gamma_{\ell-1},\gamma_{\ell+1}) - d_{\gamma_\ell}(\gamma_{\ell-1},\gamma_i) - d_{\gamma_\ell}(\gamma_{\ell+1},\gamma_j) \geq 20-3-3 = 14.\]
In particular, $\gamma_i$ and $\gamma_j$ nontrivially intersect.  Furthermore, by Proposition~\ref{P:Behrstock}, we have
\[ d_{\gamma_i}(\gamma_\ell,\gamma_j) \leq 3 \mbox{ and } d_{\gamma_j}(\gamma_i,\gamma_\ell) \leq 3,\]
as required. 

If we do not have $i < \ell - 1 < \ell < \ell + 1 < j$, then it must be that either $\ell + 1 = j$ or $\ell-1 = i$, and we can argue similarly.  For example, if $i = \ell-1$, then $\ell + 1 < j$ and by induction
\[ d_{\gamma_\ell}(\gamma_{\ell+1},\gamma_j) \leq 3 \mbox{ and } d_{\gamma_\ell}(\gamma_i,\gamma_{\ell+1}) \geq 20. \]
So $d_{\gamma_\ell}(\gamma_i,\gamma_j) \geq 17$, thus $i(\gamma_i,\gamma_j) \neq 0$, and applying Proposition~\ref{P:Behrstock} we have
\[ d_{\gamma_i}(\gamma_\ell,\gamma_j) \leq 3 \mbox{ and } d_{\gamma_j}(\gamma_i,\gamma_\ell) \leq 3 \]
as required.  The case $\ell + 1 = j$ is similar.  This completes the induction, and hence proves the claim.
\end{proof}

Observe that part (i) follows from the first part of the claim.  For part (ii), let $i < \ell < j$.  Then by the claim and the triangle inequality we have
\[ |d_{\gamma_\ell}(\gamma_i,\gamma_j) - d_{\gamma_\ell}(\gamma_{\ell-1},\gamma_{\ell+1})| \leq d_{\gamma_\ell}(\gamma_{\ell-1},\gamma_i) + d_{\gamma_\ell}(\gamma_{\ell+1},\gamma_j) \leq 6.\]
So, setting $R = R_0 +6$, part (ii) of the lemma follows from Inequality (\ref{E:twisting nearly k}).

To prove part (iii), we first prove\\

\noindent {\bf Claim.}  For any $j \in \mathbb Z$, the curves $\gamma_{j+1},\gamma_{j+2},\ldots,\gamma_{j+n}$ fill $S$.

\begin{proof}  By applying an appropriate power of $f$, and cyclically permuting the original indices $1,2,\ldots,n$, it suffices to prove that $\gamma_1,\ldots,\gamma_n$ fill $S$.  For this, we show that for any $1 \leq j \leq n$, the subsurface $X_j$ filled by $\gamma_1,\ldots,\gamma_j$ is the same as the subsurface $Z_j$ filled by $c_1,\ldots,c_j$.  We do this by induction on $j$.

The base case is $j = 1$, and then $\gamma_1 = c_1$, so $X_1 = Z_1$ is the annular neighborhood.  Now suppose that $X_{j-1} = Z_{j-1}$ for some $j \geq 2$ and we prove $X_j = Z_j$.  First observe that
\[ f_{j-1} = T_{c_1}^{k_i} \cdots T_{c_{j-1}}^{k_{j-1}} \]
is supported on $Z_{j-1} = X_{j-1}$ since $c_1,\ldots,c_{j-1}$ are contained in $Z_{j-1}$.  If $c_j \subset Z_{j-1}$, then $Z_j = Z_{j-1}$, while on the other hand
\[ \gamma_j = f_{j-1} T_{c_j}^{k_j}(c_j) = f_{j-1}(c_j) \subset Z_{j-1} = X_{j-1}\]
and hence $X_j = X_{j-1} = Z_{j-1} = Z_j$.  Thus if $c_j \subset Z_{j-1}$, we are done.  So, suppose $c_j \not \subset Z_{j-1}$.  Then $Z_j$ is determined by $Z_{j-1}$ and the isotopy classes of arcs of $c_j - Z_{j-1}$ in $S - Z_{j-1}$.  We will be done if we can show that these isotopy classes of arcs are the same as those of $\gamma_j - X_{j-1}$ in $S - X_{j-1} = S - Z_{j-1}$.  For this, observe that as above $\gamma_j = f_{j-1}(c_j)$, and since $f_{j-1}$ is supported on $X_{j-1} = Z_{j-1}$, $f_{j-1}$ cannot change the isotopy classes of arcs of $c_j - Z_{j-1}$.  Hence $\gamma_j - X_{j-1} = \gamma_j- Z_{j-1}$ is isotopic to $c_j - Z_{j-1}$, as required.  This proves the claim.
\end{proof}

First observe that by $f$--invariance, if $|j-i| \leq n$, then $d(\gamma_i,\gamma_j) \leq A'$ for some constant $A'$.  In particular, $d(\gamma_i,\gamma_j) \leq A' |j-i|$.  By the triangle inequality, this holds for all $i,j$.
%Now suppose $i < j$.  Define a path $\tau \colon [i,j] \to \C(S)$ from $\gamma_i$ to $\gamma_j$ so that for each integer $i < \ell < j$, $\tau|[\ell,\ell+1]$ is a linear parameterization of a geodesic segment from $\gamma_\ell$ to $\gamma_{\ell+1}$.  By $f$--invariance of the set $\{\gamma_r\}_{r \in \mathbb Z}$, there is a uniform bound on the lengths of those segments and thus there exists a $C> 0$ so that $\tau$ is $C$--Lipschitz, independent of $i$ and $j$.  To prove that $\tau$ is a uniform quasi-geodesic, it suffices therefore to prove that $d(\gamma_i,\gamma_j)$ is bounded below by a uniform linear function of $j-i$.

Consider any geodesic $\sigma$ in $\C(S)$ from $\gamma_i$ to $\gamma_j$ and list the vertices consecutively as $\gamma_i = \alpha_0 , \alpha_1,\ldots,\alpha_r = \gamma_j$ from $\gamma_i$ to $\gamma_j$.  By our choice of $K$, $d_{\gamma_\ell}(\gamma_i,\gamma_j) > M$, for all $i < \ell < j$.  So by Proposition~\ref{P:BGI} there is a vertex $\alpha_s$ of $\sigma$ which is disjoint from $\gamma_\ell$.  There may be more than one, but there can be at most $3$ since $\sigma$ is a geodesic (if there were more than three, two would be distance at least $3$ apart, which is impossible since they are distance $1$ from $\gamma_\ell$).  For each such $\ell$, let $\alpha_{s(\ell)}$ be the vertex closest to $\gamma_j$ which is disjoint from $\gamma_\ell$.  As in \cite[Lemma~4.4]{BBKL}, $s(\ell) \leq s(\ell')$ if $\ell \leq \ell'$.  On the other hand, since every $n$ consecutive curves fill, we have $s(\ell) < s(\ell+n)$.  Consequently, the number of vertices in $\sigma$ between $\gamma_i$ a
 nd $\gamma_j$ is at least $\frac{j-1}n$ and hence the distance is at least
\[ d(\gamma_i,\gamma_j) \geq  \frac{j-i}n  -1. \]
This provides the desired lower bound, and hence $\{\gamma_j\}$ is a uniform quasi-geodesic.

Finally, for part (iv), we note that since $\{\gamma_j\}_{j \in \mathbb Z}$ is a quasi-geodesic, and is $f$--invariant, $f$ must be pseudo-Anosov, and we have
\[ \displaystyle{\lim_{j \to \pm \infty} \gamma_j = \mu^\pm},\]
in the Hausdorff topology on $S$, after throwing away any isolated leaves of the limit.  Therefore, for every $\ell \in \mathbb Z$, every arc of $\pi_{\gamma_\ell}(\mu^+) \cup \pi_{\gamma_\ell}(\mu^-)$ is a limit of arcs in $\pi_{\gamma_\ell}(\gamma_j) \cup \pi_{\gamma_\ell}(\gamma_{-j})$, as $j$ tends to infinity.  Since some limits of arcs in the latter set can disappear (since isolated leaves of the Hausdorfff limits are discarded), the difference in diameters between the former and latter sets (for $j$ sufficiently large) is at most $2$.  Part (iv) now follows from part (ii).
\end{proof}

Now suppose $c_1,\ldots,c_n$ are as above, $\kappa_1,\ldots,\kappa_n \in \{ \pm 1\}$, and $m \geq K$, with $K$ as in Lemma~\ref{L:curves with big twist}.  Let $k_j(m) = \kappa_j m$ for $1 \leq j \leq n$, and extend this to $\{k_j(m)\}_{j \in \mathbb Z}$ as above.   Construct a sequence of homeomorphisms $\{f_m \colon S \to S \}_{m=1}^\infty$ by 
\begin{equation} \label{E:type of pA} f_m = T_{c_1}^{k_1(m)} T_{c_2}^{k_2(m)} \cdots T_{c_n}^{k_n(m)}.
\end{equation}
% and so% Since $|k_j(m)| \to \infty$, part (2) implies
%\begin{equation} \label{E:twisting nearly equals distance} d_{\gamma_j(m)}(\mu^+(m),\mu^-(m)) \geq |k_j(m)| - R
% \lim_{m \to \infty} \frac{d_{\gamma_j(m)}(\mu^+(m),\mu^-(m)}{|k_j(m)|} = 1. 
%\end{equation}

\begin{proposition} \label{P:modulus bounds for a sequence}  Let $\{f_m \colon S \to S\}_{m = K}^\infty$ be a sequence of pseudo-Anosov homeomorphisms defined as in Equation (\ref{E:type of pA}), $Q(m) = \{q_t(m)\}$ the associated flat metrics, and $\{\gamma_j(m)\}_{j \in \mathbb Z}$ the associated $f_m$--invariant collection of curves, for each $m$.  Then for all $j$, 
\[ M(\gamma_j(m),Q(m)) \geq \frac{m-R-6}2, \]
where $R$ is the constant from Lemma~\ref{L:curves with big twist}.
%\[ \lim_{m \to \infty} \frac{2M(\gamma_j(m),Q(m))}{|k_j(m)|} = 1.\]
Furthermore, there is a constant $D > 0$ so that for any $m$ and curve $\gamma \not\in \{\gamma_j(m)\}_{j \in \mathbb Z}$,
\[ M(\gamma,Q(m)) \leq D. \]
\end{proposition}
\begin{proof}  Let $\mu^\pm(m)$ denote the stable/unstable foliations of $f_m$.  Since $|k_j(m)| = m \geq K$,  $\{\gamma_j(m)\}_{j \in \mathbb Z}$ satisfies the conclusion of the Lemma~\ref{L:curves with big twist}.  Combining this with Proposition~\ref{P:flat annuli iff twisting} we have
\[ M(\gamma_j(m),Q(m)) \geq \frac{d_{\gamma_j(m)}(\mu^+(m),\mu^-(m))}{2} - 3 \geq \frac{m-R-6}2.\]
%Since $k_j(m) \to \infty$ for all $j$,  (\ref{E:twisting nearly equals distance}) combined with Proposition~\ref{P:flat annuli iff twisting} and  implies
%\[ \lim_{m \to \infty} \frac{2M(\gamma_j(m),Q(m))}{|k_j(m)|} = \lim_{m \to \infty} \frac{M(\gamma_j(m),Q(m))}{d_{\gamma_j(m)}(\mu^+(m),\mu^-(m))/2} = 1.\]
This proves the first statement.

Let $X_{f_m}$ denote the mapping torus of $f_m$ equipped with its hyperbolic metric, and $\widetilde X_{f_m}$ the cover of $X_{f_m}$ corresponding to the fiber subgroup $\pi_1(S)$.  Appealing to
%%AR added to
the {\em Short Curve Theorem} of  Minsky
%%AR changed to of Minsky
 \cite{MinskyShort2} (see also the {\em Length Bound Theorem} from Brock-Canary-Minsky's \cite{BCM-ELC2}), the curves $\gamma_j(m)$ all have length in $\widetilde X_{f_m}$ tending to zero as $m$ tends to infinity.  Being $f_m$--invariant, they push down to $n$ closed geodesics in $X_{f_m}$.

%%AR removed the sidenote
%\cl{Maybe get better references for this.. Hodgson-Kerckhoff?}
The geometric limit of the sequence of hyperbolic $3$--manifolds $X_{f_m}$ is the cusped hyperbolic $3$--manifold obtained by drilling out the $n$ curves, realized on $n$ different fibers of $X_{f_m}$ (see \cite{thurstonnotes}) and $X_{f_m}$ is obtained from $X_\infty$ by $(1,k_j(m))$--Dehn filling on $X_\infty$ for all $m > 0$ as in \cite{LoMo}.
The geometric convergence ensures that there is a uniform lower bound to the length of any curve in $X_{f_m}$ which is not one of the $n$ curves, and hence there is a uniform lower bound (independent of $m$) to the length of any curve $\gamma$ in $\widetilde X_{f_m}$ which is not in $\{\gamma_j(m)\}_{j \in \mathbb Z}$.  By the Short Curve Theorem again, it follows that $d_{\gamma}(\mu^+(m),\mu^-(m))$ is uniformly bounded, independent of $m$ and $\gamma$.  By Propostion~\ref{P:flat annuli iff twisting}, the modulus $M_t(\gamma)$ of any $q_t(m)$--Euclidean cylinder with core curve isotopic to $\gamma$ is uniformly bounded, independent of $m$ and $\gamma$, as required.
\end{proof}

The following provides a useful mechanism for deciding when a pseudo-Anosov $f \colon S \to S$ constructed as above is not a virtual lift.

\begin{theorem} \label{T:lifting curves} Suppose $\{f_m \colon S \to S\}_m$ is a sequence of pseudo-Anosov homeomorphisms as in Equation (\ref{E:type of pA}), $\{\{\gamma_j(m)\}_{j \in \mathbb Z}\}_m$ are the associated sequences of curves, and that the stretch factors $\lambda(f_m)$ have degree greater than $2$ over $\mathbb Q$.  Then there exists a positive integer $N \geq K$, so that if $m \geq N$ and $f_m$ is a virtual lift of some $\phi_m \colon \Sigma_m \to \Sigma_m$ via a branched covering $p_m \colon S \to \Sigma_m$, then there are representatives of the curves $\gamma_j(m)$ so that $p_m^{-1}(p_m(\gamma_j(m))) = \gamma_j(m)$ for all $j$.
\end{theorem}

\begin{proof}   Suppose that $p_m \colon S \to \Sigma_m$ is a branched covering and $\phi_m$ a map that lifts to a power of $f_m$.  Since $\lambda(f_m)$ is not quadratic irrational, $\Sigma_m$ is not a sphere with four marked points or a torus.  Let $\Xi(m) = \{\xi_t(m)\}$ and $Q(m) = \{q_t(m)\}$ be the associated compatible family of flat metrics.  By Lemmas~\ref{L:cylinders into orientable} and \ref{L:cylinders into nonorientable}, for each $j$, we can choose a representative of $\gamma_j(m)$ so that $p_m(\gamma_j(m))$ is a cylinder curve with
\[ M(p_m(\gamma_j(m)),\Xi(m)) \geq \frac{M(\gamma_j(m),Q(m))}2.\]
On the other hand, by the Riemann-Hurwitz Theorem and Poincar\'e Hopf Theorem, there is a bound $d$ on the degree of $p_m$.
Let $B = B(S,d)$ be the constant from Lemma~\ref{L:preimage of cylinders}.  Then
%For any $j$ and $m$, consider the metrics $\xi_t(m)$ and $q_t(m)$, where $t = t_{\gamma_j(m)}$ is the balance time for $\gamma_j(m)$ (and hence for $p_m(\gamma_j(m))$).  
there is a sub-cylinder $Y_j(m)$ of the cylinder about $p_m(\gamma_j(m))$ so that each component of $\widetilde Y_j(m) = p_m^{-1}(Y_j(m))$ is a Euclidean cylinder and has maximal modulus at least
\[ B M(p_m(\gamma_j(m)),\Xi(m)) \geq \frac{BM(\gamma_j(m),Q(m))}2.\]
One component of $\widetilde Y_j(m)$ is contained in the original cylinder with core curve $\gamma_j(m)$.  Without loss of generality, we may choose $\gamma_j(m)$ so that $p_m(\gamma_j(m)) \subset Y_j(m)$, and hence $\gamma_j(m) \subset p_m^{-1}(p_m(\gamma_j(m))) \subset \widetilde Y_j(m)$.

Let $D  > 0$ be the constant from Proposition~\ref{P:modulus bounds for a sequence}.  We choose $N> K$ so that if $m \geq N$, then for all $j$
\[ \frac{B(m-R-6)}4 > D.\]
Then if $\gamma_j(m)'$ is any component of $p_m^{-1}(p_m(\gamma_j(m)))$, and $\widetilde Y'_j(m) \subset \widetilde Y_j(m)$ is the component containing it, then the bound above on the maximal modulus of $\widetilde Y'_j(m)$ combined with Proposition~\ref{P:modulus bounds for a sequence} implies
\[ M(\gamma_j(m)',Q(m)) \geq \frac{BM(\gamma_j(m),Q(m))}2 \geq \frac{B(m-R-6)}4 > D.\]
Consequently, $\gamma_j(m)'$ must be one of the curves $\gamma_{j'}(m)$.  However, the direction of $\gamma_j(m)$ and $\gamma_j(m)'$ in the Euclidean cone metric are the same, while if $j' \neq j$, the curves $\gamma_{j'}(m)$ and $\gamma_j(m)$ intersect nontrivially by Lemma~\ref{L:curves with big twist}.  Therefore, $\gamma_j(m)'$ and $\gamma_j(m)$ must either be equal or isotopic.

Thus, all the components of $p_m^{-1}(p_m(\gamma_j(m)))$ are isotopic to $\gamma_j(m)$, and are hence contained in a single cylinder.  By Lemma~\ref{L:cylinders into orientable}, either $p_m$ restricted to this cylinder is a covering map---in which case, $p_m^{-1}(p_m(\gamma_j(m))) = \gamma_j(m)$, and we are done---or the image of the cylinder is a half-pillowcase.  If the latter happens, then we take $\gamma_j(m)$ to be the unique core curve in the cylinder that projects to the core geodesic segment of the half-pillowcase, we get $p_m^{-1}(p_m(\gamma_j(m))) = \gamma_j(m)$, as required.
\end{proof}

\begin{corollary} \label{C:at worst hyperellipticish} In addition to the assumptions from Theorem~\ref{T:lifting curves}, suppose that $i(c_i,c_{i+1}) = 1$ for some $i$.   If $m \geq N$, and $f_m$ is a virtual lift of some $\phi_m \colon \Sigma_m \to \Sigma_m$ via a branched covering $p_m \colon S \to \Sigma_m$, then either $p_m$ is a homeomorphism, or $\Sigma_m$ is the quotient by an orientation preserving involution preserving the isotopy classes of $c_1,\ldots,c_m$.
\end{corollary}

\begin{proof} Choose representatives $\gamma_j(m)$ for the isotopy classes, for all $j$, as in the theorem.  Note that $i(\gamma_i(m),\gamma_{i+1}(m)) = i(c_i,c_{i+1}) = 1$.  Since $p_m^{-1}(p_m(\gamma_j(m))) = \gamma_j(m)$ for all $j$, it follows that if $x = \gamma_i(m) \cap \gamma_{i+1}(m)$, then $p_m^{-1}(p_m(x)) = \{x\}$.  Since the image of the cylinders about $\gamma_j(m)$ are either cylinders or half-pillowcases, the local degree of $p_m$ near $x$ must be $1$ or $2$.  Thus either $p_m$ is a homeomorphism or else it has degree $2$.  In the latter case, the branched covering is regular and the covering group is generated by an orientation preserving involution $\tau$.

Suppose now that $p_m$ has degree $2$.  Since $p_m^{-1}(p_m(\gamma_j(m))) = \gamma_j(m)$ for all $j$, it follows that $\tau(\gamma_j(m)) = \gamma_j(m)$.  We now show that $\tau(c_j) = c_j$ for each $j = 1,\ldots, n$.  For $j = 1$, note that $c_1= \gamma_1(m)$.  We use this as the base case for induction.  Assuming $\tau(c_j)  = c_j$ for all $1 \leq j \leq \ell < n$, we prove that $\tau(c_{\ell+1}) = c_{\ell+1}$.  For this, observe that
\[ \gamma_{\ell}(m) = T_{c_1}^{k_1(m)} \cdots T_{c_\ell}^{k_\ell(m)} T_{c_{\ell+1}}^{k_{\ell+1}(m)}(c_{\ell + 1}) = T_{c_1}^{k_1(m)} \cdots T_{c_\ell}^{k_\ell(m)}(c_{\ell + 1}) \]
since $T_{c_{\ell+1}}$ fixes $c_{\ell+1}$.  Therefore, we have
\[ T_{c_\ell}^{-k_\ell(m)} \cdots T_{c_1}^{-k_1(m)} (\gamma_\ell(m) ) = c_{\ell+1}.\]
Since $\tau$ preserves each of $c_1,\ldots,c_\ell$, it commutes with $T=T_{c_\ell}^{-k_\ell(m)} \cdots T_{c_1}^{-k_1(m)}$, thus the equations $T(\gamma_\ell(m)) = c_{\ell+1}$ and $\tau(\gamma_\ell(m)) = \gamma_\ell(m)$ imply
\[ \tau(c_{\ell+1}) = \tau T(\gamma_\ell(m))  = \tau T \tau^{-1} \tau(\gamma_\ell(m)) = T(\gamma_\ell(m)) = c_{\ell+1}.\]
This completes the proof.
\end{proof}

%%%%%%%%%%%%%%%%%%%%%%%%%%%%%%%%%%%%
%%%%%%%%%%%%%%%%%%%%%%%%%%%%%%%%%%%%
\subsection{Strenner's construction}
%%%%%%%%%%%%%%%%%%%%%%%%%%%%%%%%%%%%
%%%%%%%%%%%%%%%%%%%%%%%%%%%%%%%%%%%%

The key to obtaining the required degree for the dilatation is the following special case of a result of Strenner \cite[Theorem~5.3]{Strenner}, building on a theorem of Penner \cite{PennerCon}.

\begin{theorem}[Strenner] \label{T:Strenner}  Suppose $A = a_1 \cup \ldots \cup a_n$ and $B = b_1 \cup \cdots \cup b_n$ are multicurves that fill the surface $S$, and let $N = (i(a_i,b_j))_{ij}$ be the matrix of intersection numbers and $G$ the associated (bipartite) adjacency graph (with a vertex for every $a_i$ and every $b_j$ and an edge between $a_i$ and $b_j$ if $i(a_i,b_j) \neq 0$).  Suppose
\begin{enumerate}
\item $rk(N)  = r > 0$,
\item $a_{i_1} b_{i_1} a_{i_2} b_{i_2} \cdots a_{i_d} b_{i_d} a_{i_1}$ are the vertices of a closed, contractible loop in $G$ visiting every vertex, and
\end{enumerate}
Then for all $m > 0$ sufficiently large, the mapping classes
\[ f_m = T_{a_{i_1}}^mT_{b_{i_1}}^{-m} \cdots T_{a_{i_d}}^m T_{b_{i_d}}^{-m} \]
%%AR changed is to are
are pseudo-Anosov and $\lambda(f_m)$ has degree $2r$.
\end{theorem}

%Observe that adjacent vertices of the adjacency graph correspond to curves that intersect non-trivially.  Thus setting
%\[ c_{2\ell+1} = a_{i_\ell} \quad \mbox{ and } \quad c_{2\ell+2} \]
%for $0 \leq \ell \leq d$, the curves $c_1,\ldots,c_{2d+2}$ have $i(c_j,c_{j+1}) \neq 0$ for all $j$ (with indices taken modulo $2d+2$).  

%%%%%%%%%%%%%%%%%%%%%%%%%%%%%%%%%%%%
%%%%%%%%%%%%%%%%%%%%%%%%%%%%%%%%%%%%
\section{Proof of the main theorem}
%%%%%%%%%%%%%%%%%%%%%%%%%%%%%%%%%%%%
%%%%%%%%%%%%%%%%%%%%%%%%%%%%%%%%%%%%

We will apply the results of the preceding section to a pair of multicurves, depending on the even integer $r > 2$, for every sufficiently large genus surface.  For this, we start with a particular pair of simple closed curves $a,b$ that fill a genus $3$ surface $X$ with one boundary component and intersect in exactly $5$ points with the same sign (after orienting them appropriately).    This pair is described in Figure~\ref{F:genus 3 curves}.

\begin{figure}[h]
\begin{center}
\begin{tikzpicture}[scale = 1]
\draw[->] (0,0) --  (.5,0);
\draw[->] (.5,0) --  (2,0);
\draw[->] (2,0) --  (4,0);
\draw[->] (4,0) --  (6,0);
\draw[->] (6,0) --  (8,0);
\draw[->] (8,0) --  (9.5,0);
\draw (9.5,0) --  (10,0);
\draw[->] (1,-1) -- (1,.5);
\draw[->] (3,-1) -- (3,.5);
\draw[->] (5,-1) -- (5,.5);
\draw[->] (7,-1) -- (7,.5);
\draw[->] (9,-1) -- (9,.5);
\draw (1,.5) -- (1,1);
\draw (3,.5) -- (3,1);
\draw (5,.5) -- (5,1);
\draw (7,.5) -- (7,1);
\draw (9,.5) -- (9,1);
\node at (1,1.3) {$\beta_2$};
\node at (3,1.3) {$\beta_4$};
\node at (5,1.3) {$\beta_1$};
\node at (7,1.3) {$\beta_5$};
\node at (9,1.3) {$\beta_3$};
\node at (1,-1.3) {$\beta_1$};
\node at (3,-1.3) {$\beta_2$};
\node at (5,-1.3) {$\beta_3$};
\node at (7,-1.3) {$\beta_4$};
\node at (9,-1.3) {$\beta_5$};
\node at (.5,-.3) {$\alpha_1$};
\node at (2,-.3) {$\alpha_2$};
\node at (4,-.3) {$\alpha_3$};
\node at (6,-.3) {$\alpha_4$};
\node at (8,-.3) {$\alpha_5$};
\node at (9.5,-.3) {$\alpha_1$};
\node at (5.2,.2) {$x$};
\draw[fill=black] (5,0) circle (.05cm);
\draw plot [smooth] coordinates {(0,.55)(.4,.6)(.5,1)};
\draw plot [smooth] coordinates {(0,-.55)(.4,-.6)(.5,-1)};
\draw plot [smooth] coordinates {(1.5,1)(1.6,.6)(2.4,.6)(2.5,1)};
\draw plot [smooth] coordinates {(3.5,1)(3.6,.6)(4.4,.6)(4.5,1)};
\draw plot [smooth] coordinates {(5.5,1)(5.6,.6)(6.4,.6)(6.5,1)};
\draw plot [smooth] coordinates {(7.5,1)(7.6,.6)(8.4,.6)(8.5,1)};
\draw plot [smooth] coordinates {(1.5,-1)(1.6,-.6)(2.4,-.6)(2.5,-1)};
\draw plot [smooth] coordinates {(3.5,-1)(3.6,-.6)(4.4,-.6)(4.5,-1)};
\draw plot [smooth] coordinates {(5.5,-1)(5.6,-.6)(6.4,-.6)(6.5,-1)};
\draw plot [smooth] coordinates {(7.5,-1)(7.6,-.6)(8.4,-.6)(8.5,-1)};
\draw plot [smooth] coordinates {(10,.55)(9.6,.6)(9.5,1)};
\draw plot [smooth] coordinates {(10,-.55)(9.6,-.6)(9.5,-1)};
\draw (0,.55) -- (0,-.55);
\draw (.5,1) -- (1.5,1);
\draw (2.5,1) -- (3.5,1);
\draw[line width=2] (4.5,1) -- (5.5,1);
\draw (6.5,1) -- (7.5,1);
\draw (8.5,1) -- (9.5,1);
\draw (10,.55) -- (10,-.55);
\draw[line width=2] (.5,-1) -- (1.5,-1);
\draw (2.5,-1) -- (3.5,-1);
\draw (4.5,-1) -- (5.5,-1);
\draw (6.5,-1) -- (7.5,-1);
\draw (8.5,-1) -- (9.5,-1);
\end{tikzpicture}
\caption{The curves $a$ and $b$ are cut into arcs $a = \alpha_1 \cup \cdots \cup \alpha_5$ and $b = \beta_1 \cup \cdots \cup \beta_5$ at the points of intersection $a \cap b$.  The surface $X$ of genus $3$ with one boundary component is shown, cut open along essential arcs meeting each of the arcs $\beta_1,\ldots,\beta_5$ and $\alpha_1$ as labelled.  The point $x$ is the fixed point of an involution $\rho$ of $X$ leaving each of $a$ and $b$ invariant.  The thick line represents an essential arc $\delta$ meeting $b$ in the arc $\beta_1$.}
\label{F:genus 3 curves}
\end{center}
\end{figure}
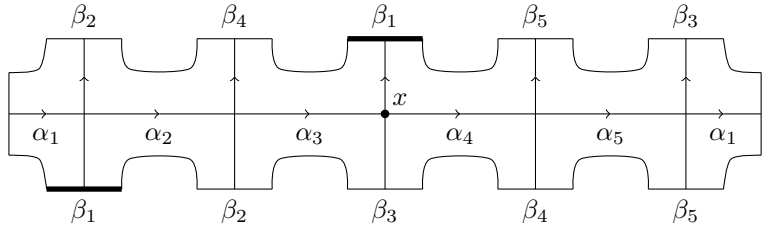

\begin{lemma} \label{L:special genus 3 involution}
Up to isotopy, the surface $X$ admits exactly one orientation preserving involution $\rho$ leaving both $a$ and $b$ invariant.
\end{lemma}
\begin{proof}
Let $\rho \colon X \to X$ denote the ``obvious'' involution of $X$, evident in Figure~\ref{F:genus 3 curves}, given by rotation about the point $x$---it is straightforward to check that the rotation extends over the gluing of the arcs in the reconstruction of $X$.  To see that $\rho$ is the only orientation preserving involution preserving $a$ and $b$, we note that such an involution would define a graph automorphism of $a \cup b$, viewed as a four-valent graph with $5$ vertices, and would preserve the cyclic ordering around each vertex.  Any such nontrivial graph automorphism would necessarily fix one of the vertices, and would be determined by which vertex it fixed. It is now easy to show that the only such nontrivial graph automorphism is $\rho$.
\end{proof}

We now prove the following theorem, which implies the Main Theorem in the introduction.

\begin{theorem} For each $r > 1$ and closed orientable surface $S = S_g$ with $g \geq r+2$, there exists a pseudo-Anosov 
%%AR corrected the spelling of homeomorphism
homeomorphism $f \colon S \to S$ with stretch factor $\lambda(f)$ of degree $2r$ and orientable foliations that is not a virtual lift.
\end{theorem}
\begin{proof} Embed the surface $X$ of genus $3$ with one boundary component as an essential subsurface of $S_g$.  The complement $Z$ is a surface of genus $g-3$ with one boundary component.  
Let $a_1 = a$ and $b_1 = b$ as constructed above.  The arc $\delta$ from Figure~\ref{F:genus 3 curves} can be connected to an arc $\delta'$ in $Z$ to construct an essential simple closed curve we denote $a_2$, that has intersection number $1$ with $b_1$ and $0$ with $a_1$.   

If $r =2$, then we choose any essential simple closed curve $b_2$ in $Z$ which fills with $\delta'$ so that all $k$ intersection points have the same sign.  The intersection matrix is
\[ (i(a_i,b_j)) = \left( \begin{array}{cc} 5 & 1\\ 0 & k \end{array} \right).\]
This has rank $2$.  Now consider the sequence of mapping classes defined by:
%%AR rewrote this a bit
\[ f_m = T_{a_1}^mT_{b_1}^{-m}T_{a_2}^mT_{b_2}^{-m}T_{a_2}^m T_{b_1}^{-m}.\]
On the one hand, the sequence $\{f_m\}_{m=1}^\infty$ satisfies the hypothesis of Theorem~\ref{T:Strenner}, and so for $m$ sufficiently large, the $f_m$ are pseudo-Anosov, and have stretch factors
$\lambda(f_m)$ having degree $4$ over $\mathbb Q$.  On the other hand, consecutive curves in the sequence $a_1,b_1,a_2,b_2,a_2,b_1$ intersect nontrivially
%%AR reordered the previous two words
(cyclically), and $i(b_1,a_2) = 1$.  Consequently, the sequence $\{f_m\}$ also satisfies Corollary~\ref{C:at worst hyperellipticish}, and so by taking $m$ larger if necessary, it follows that if $f_m$ is a virtual lift via a branched covering $p_m \colon S \to \Sigma_m$, then $p_m$ has degree two, and $\Sigma_m$ is the quotient by an orientation preserving involution $\tau$ preserving $a_1,b_1,a_2,b_2$. The involution $\tau$ must restrict to $\rho$ on $X$ (up to isotopy) by Lemma~\ref{L:special genus 3 involution}.  However, $\rho$ does not preserve the isotopy class of $\delta$ in $X$, and so $\tau$ cannot preserve $b_1$, a contradiction.
Therefore, there is no such involution $\tau$, and hence $f$ is not a virtual lift.  The curves can be oriented so the intersection points have all the same sign, and hence invariant foliations are orientable.

If $r  >2$, we proceed in a similar fashion, choosing a curve $b_2$ that intersects $a_2$ once and is disjoint from all other curves, $a_3$ intersecting $b_2$ once and disjoint from all other curves, and continuing until we have constructed
\[ a_1,b_1,a_2,b_2,\ldots,a_{r-1},b_{r-1},a_r.\]
Then we finally choose $b_r$ so that the union of all the curves fills $S$ and so that $b_r$ is disjoint from all curves except $a_r$, which it intersects in $k$ points, for some $k > 0$, all with the same sign.  The $r \times r$ intersection matrix now has the form
\[ (i(a_i,b_j)) = \left( \begin{array}{cccccccccccccc}  5 & 1 & 0 & 0 &  \cdots & 0 & 0 & 0\\ 
0 & 1 & 1 & 0 &  \cdots & 0 & 0 & 0\\ 
0 & 0 & 1 & 1 &  \cdots & 0 & 0 & 0\\ 
0 & 0 & 0 & 1 &  \cdots & 0 & 0 & 0\\ 
\vdots & \vdots  & \vdots & \vdots & \ddots & \vdots & \vdots & \vdots \\
0 & 0 & 0 & 0  & \cdots & 1&  1 & 0\\ 
0 & 0 & 0 & 0  & \cdots & 0 & 1 & 1\\ 
0 & 0 & 0 &  0 & \cdots & 0 & 0 & k\\ 
\end{array} \right) \]
This has determinant $5k$, and so has rank $r$. 
%%AR rewrote this part a bit
As above, we can now use Theorem~\ref{T:Strenner} and Corollary~\ref{C:at worst hyperellipticish} to construct a sequence of pseudo-Anosov homeomorphisms
\[ f_m = T_{a_1}^mT_{b_1}^{-m} \cdots T_{a_r}^mT_{b_r}^{-m}T_{a_r}^m \cdots T_{a_2}^m T_{b_1}^{-m},\]
and arguing exactly as in the case $r = 2$ to deduce that for $m$ sufficiently large $\lambda(f_m)$ has degree $2r$ over $\mathbb Q$, and that $f_m$ is not a virtual lift.
\end{proof}

%%AR added the bibliography to the file

%\bibliographystyle{alpha}
 %\bibliography{main}

\end{document}